\documentclass[a4paper,10pt]{amsart}
\usepackage[utf8]{inputenc}
\usepackage{amsmath,amsthm,amssymb}
\usepackage{enumerate}
\usepackage{comment}

\usepackage[colorinlistoftodos,prependcaption,textsize=tiny]{todonotes}

\newtheorem{lemma}{Lemma}[section]
\newtheorem{proposition}[lemma]{Proposition}
\newtheorem{theorem}[lemma]{Theorem}
\newtheorem{corollary}[lemma]{Corollary}
\newtheorem{conj}[lemma]{Conjecture}
\newtheorem{question}[lemma]{Question}

\theoremstyle{definition}
\newtheorem{definition}[lemma]{Definition}
\newtheorem{example}[lemma]{Example}
\newtheorem{remark}[lemma]{Remark}

\renewcommand{\Re}{\operatorname{Re}}
\renewcommand{\Im}{\operatorname{Im}}

\def\ffi{\varphi}
\def\dst{\displaystyle}

\DeclareMathOperator{\Id}{Id}

\def\bb{{\mathcal B}}
\def\cc{{\mathcal C}}

\def\ee{{\mathcal E}}
\def\ff{{\mathcal F}}

\def\mm{{\mathcal M}}

\def\C{{\mathbb{C}}}

\def\N{{\mathbb{N}}}

\def\Q{{\mathbb{Q}}}
\def\R{{\mathbb{R}}}

\def\T{{\mathbb{T}}}

\def\Z{{\mathbb{Z}}}

\newcommand{\abs}[1]{{\left|{#1}\right|}}
\newcommand{\scal}[1]{{\left\langle{#1}\right\rangle}}

\usepackage[colorinlistoftodos]{todonotes}

\begin{document}
\title[Uniqueness of phase retrieval from three measurements]{Uniqueness of phase retrieval from three measurements}

\author{Philippe Jaming}
\email{philippe.jaming@math.u-bordeaux.fr}

\author{Martin Rathmair}
\email{martin.rathmair@math.u-bordeaux.fr}
\address{Univ. Bordeaux, IMB, UMR 5251, F-33400 Talence, France. CNRS, IMB, UMR 5251, F-33400 Talence, France.}


\begin{abstract}
In this paper we consider the question of finding an as small as possible family of operators
$(T_j)_{j\in J}$ on $L^2(\R)$ that does phase retrieval: every $\ffi$ is uniquely determined 
(up to a constant phase factor)
by the phaseless data $(|T_j\ffi|)_{j\in J}$. This problem arises in various fields of applied sciences
where usually the operators obey further restrictions. 

Of particular interest here are so-called {\em coded diffraction paterns} where the operators are of the form
$T_j\ffi=\ff[m_jf]$, $\ff$ the Fourier transform and $m_j\in L^\infty(\R)$ are ``masks''.
Here we explicitely construct three real-valued masks $m_1,m_2,m_3\in L^\infty(\R)$
so that the associated coded diffraction patterns
do phase retrieval. This implies that the three self-adjoint operators $T_j\ffi=\ff[m_j\ff^{-1}\ffi]$
also do phase retrieval. The proof uses complex analysis.

We then show that some natural analogues of these operators
in the finite dimensional setting do not always lead to the same uniqueness result
due to an undersampling effect.
\end{abstract}

\maketitle

\section{Introduction}
Generally speaking, phase retrieval refers to the problem of recovering a signal from phaseless linear measurements. 
Typical instances of phase retrieval tasks include the question of recovering a function from the magnitude of its Fourier transform or a variant therof.
Such problems arise in various areas of natural sciences ranging
from signal processing to quantum mechanics. This family of problems has recently attracted a lot of attention in the mathematical community, and we refer {\it e.g.} to \cite{BBE,grohsstable20,GKK,JEH,KST,She} for an overview of some recent developments, as well as for references to concrete problems.

\subsection{Problem setting}
Within this article we will predominantly deal with signals $f$ of one real variable, i.e. $f\in L^2(\mathbb{R})$.
The phase retrieval problems we shall consider will be associated to a given family of operators.

\begin{definition}\label{def:uniqueness}
 Let $\mathcal{T}=(T_j)_{j\in J}$ be a family of linear operators on $L^2(\mathbb{R})$, i.e. $T_j:L^2(\mathbb{R})\rightarrow \mathbb{C}^{\Omega_j}$ linear. We say that $\mathcal{T}$ \emph{does phase retrieval} if 
 \[
  \phi,\psi \in L^2(\mathbb{R}):~ |T_j\psi|=|T_j\phi|,~j\in J 
  \quad\Rightarrow\quad
  \exists c\in \mathbb{T}:  ~\text{s.t.}~ \psi = c \phi
 \]
 with $\mathbb{T}$ the set of complex numbers of modulus $1$.
\end{definition}
Clearly, due to the linearity of the operators, $\phi$ and $c\phi$  produce the same phaseless measurements when $|c|=1$.
Thus, the notion of uniqueness introduced in Definition \ref{def:uniqueness} is the best one can hope for.\\
In practice, an arbitrary linear operator $T$ will in general not represent an attainable measurement. Moreover, measurements may be costly resulting in natural restrictions on the number of operators employed.
To put it casually, the objective in this article is to identify families $\mathcal{T}$ which do phase retrieval subject to the two side constraints 
\begin{enumerate}[i)]
 \item the operators represent physically meaningful objects and
 \item use as few operators as possible.
\end{enumerate}
Constraint i) is admittedly very vaguely phrased, and depends strongly on the concrete application one has in mind.
In the subsequent paragraph, Section \ref{subsec:motivation} we provide some possible physical context and concretize the question further.\\
To summarize, formally the task we are confronted with is the following.
\begin{quote}
 Given a set $\mathcal{A}$ of admissible linear operators on $L^2(\mathbb{R})$, find a family of operators $\mathcal{T}\subseteq \mathcal{A}$ which does phase retrieval and which is as small as possible. 
\end{quote}
\subsection{Motivation}\label{subsec:motivation}
Next we briefly discuss two important applications in physics where the problem of lost phase information appears.
Both of these instances naturally fit into the problem formulation outlined above. 
\subsubsection{Diffractive Imaging}
Perhaps the most prominent example of a phase retrieval problem arises in diffraction imaging, where one seeks to determine an unknown object represented by $f\in L^2(\mathbb{R})$ given its so-called diffraction pattern, which is represented by $|\hat{f}|$, the modulus of its Fourier transform. 
We refer to \cite{Goo} for the derivation of this model from physical considerations.
Here and in the remainder we normalize the Fourier transform according to 
\begin{equation*}
 \hat{f}(\xi) = \ff f(\xi) := \int_{\mathbb{R}} f(x)e^{-2\pi i x\xi}\,dx,\quad \xi\in \mathbb{R},
\end{equation*}
for $f\in L^1(\mathbb{R})$, and extend the definition to $L^2(\mathbb{R})$ in the usual way.\\
The mapping $f\mapsto |\hat{f}|$ is far from injective: Clearly, for an arbitrary measurable phase function $\varphi:\mathbb{R}\rightarrow \mathbb{R}$ we get that $f_\varphi:=\ff^{-1}[e^{i\varphi}\hat{f}]$ has the same Fourier modulus as $f$.\\

There are two rather obvious strategies to overcome these issues of non-unique\-ness:

\smallskip

\noindent\textbf{Restrict the signal space:} A popular constraint is to assume the signal under consideration to be compactly supported.
 The problem then amounts to determining a band-limited function from its modulus.
 However, this additional assumption is known not to render the problem unique.
To start with, it is still possible to modulate and conjugate the function (this are called trivial solutions)
and more ambiguous solutions can be constructed by employing what is known as the ``zero-flipping'' operation. 
 For details we refer to the articles of Akutowicz\cite{Ak}, Walther\cite{Wa} and Hofstetter\cite{Ho}. 
 In particular, given a compactly supported $f\in L^2(\mathbb{R})$ there is in general a huge (uncountably infinite!) set of non-equivalent ambiguous solutions, all of which have compact support.
 Corresponding results hold true in the context of wide-banded signals, see \cite{KJP}.

\smallskip

\noindent\textbf{Collect several diffraction patterns:} The idea of this approach is to accumulate more information by acquiring several diffraction patterns making use of so-called masks. 
 In our setup, a mask would then be a function $\gamma\in L^\infty(\mathbb{R})$, which interacts multiplicatively with the unknown signal $f$ before computing its diffraction pattern $|\ff [\gamma\cdot f]|$. Measurements acquired in this manner are also known as \emph{coded-diffraction patterns} ({\it see e.g.} \cite{CLS,GKK}).
 Therefore, in this particular context of diffraction imaging it appears natural to define the set of admissible operators by
 \begin{equation}\label{def:Adi}
  \mathcal{A}_{DI}=\left\{\ff \circ m_\gamma:~ \gamma\in L^\infty(\mathbb{R})\right\},
  \end{equation}
 where $m_\gamma(f):=\gamma\cdot f$ denotes the multiplication operator, and seek for a (small) family of operators $(T_j)_{j\in J}\subseteq \mathcal{A}_{DI}$ with corresponding masks $(\gamma_j)_{j\in J}$ and which does phase retrieval.

 \subsubsection{Quantum Mechanics}
 A second motivation for looking at this kind of problem comes from quantum mechanics, and in particular from a question stemming back from the work of W. Pauli.
 The aim here is to formulate everything in terms of a mathematical language. However, we dedicate the Appendix to bridging the gap between the physics literature and our formulations.\\
 
In a footnote to the {\em Handbuch der Physik} article on the general principle
of wave mechanics \cite{Pa}, W. Pauli asked whether a wave function $\psi\in L^2(\mathbb{R})$ is uniquely determined (up to a constant phase factor) by the pair
$(|\psi|,|\widehat{\psi}|)$, which is sometimes called the {\em Pauli data}. In our terms, the question Pauli posed amounts to 
\begin{center}
\emph{
 Does $(\Id, \ff)$ do phase retrieval?}
\end{center}

The first counter-example seems to be due to Bargmann who considered
the following simple example based on complex Gaussians:

\begin{example}
Let $\psi_\pm(x)=e^{-(1\pm i)\pi x^2}$ so that $\widehat{\psi_\pm}(\xi)=
2^{-1/4}e^{\mp i\pi/8}e^{-(1\mp i)\pi \xi^2/2}$. It follows
that $\psi_+=\overline{\psi_-}$ and $\widehat{\psi_+}=\overline{\widehat{\psi_-}}$
so that $|\psi_+|=|\psi_-|$ and $|\widehat{\psi_+}|=|\widehat{\psi_-}|$.
\end{example}

However, we may introduce slightly less restrictive notions of equivalence, such as the following two:

-- two states $\ffi,\psi$ are equivalent up to a constant phase factor and {\em conjugation}, 
if there exists $\lambda\in\C$ with $|\lambda|=1$ such that either $\ffi=\lambda\psi$
or $\ffi=\lambda\overline{\psi}$

-- two states $\ffi,\psi$ are equivalent up to a constant phase factor and {\em conjugate reflection}, 
if there exists $\lambda\in\C$ with $|\lambda|=1$ such that either $\ffi=\lambda\psi$
or $\ffi=\lambda\psi^*$ where $\psi^*(x)=\overline{\psi(-x)}$.

In Bargmann's example, the states $\psi_\pm$ are still equivalent up to a constant phase factor
and conjugation (or conjugate reflection). One may then ask for a class of states $\cc\subset L^2(\R)$ such that
$\ffi,\psi\in\cc$ with same Pauli data are necessarily equivalent (eventually up to a conjugation and reflection as well). This problem has attracted some attention over the years 
and more evolved counter-examples have been found, (see \cite{Co,CH,Is,Ja1,janss2,Vo} to name a few)
some which are still equivalent up to a constant phase factor and conjugation or conjugate reflection, 
and some which are not equivalent in this less restrictive form.\\

From a quantum mechanical perspective it is natural to consider unitary operators as admissible. We once more refer to the Appendix where we elaborate on why this is a natural choice. 
To be more precise, we define 
\begin{equation}\label{def:Aqm}
 \mathcal{A}_{QM}= \left\{ T:L^2(\mathbb{R})\rightarrow L^2(\mathbb{R},\mu) ~\text{unitary}, ~\text{with}~ \mu ~\text{a Borel measure}\right\}.
\end{equation}
Obviously, $\Id$ and $\ff$ belong to this class of operators. 
A natural extension to Pauli's question is the following conjecture attributed to R. Wright (based on a degree of freedom argument) which is mentioned
in \cite{Vo}:

\begin{conj}[R. Wright]\label{conj:wright}
There exists a unitary operator $T\in \mathcal{A}_{QM}$ such that 
$(\Id, \ff, T)$ does phase retrieval.
\end{conj}

To the best of our knowledge, both Wright's conjecture as well as the following relaxed version remain open up to this point in time.
\begin{conj}\label{conj:3unitaries}
 There exist $T_1,T_2,T_3 \in \mathcal{A}_{QM}$ such that 
 $(T_1,T_2,T_3)$ does phase retrieval.
\end{conj}

If the restriction on the number of operators in Wright's conjecture is dropped we can give a positive answer:
For instance,
we may consider the fractional Fourier transform defined as follows:
for $\alpha\in\R\setminus\pi\Z$, let $c_\alpha=\dst\frac{\exp\frac{i}{2}\left(\alpha-\frac{\pi}{2}\right)}{\sqrt{|\sin\alpha|}}$
be a square root of $1-i\cot\alpha$. For $u\in L^1(\R)$ and $\alpha\notin\pi\Z$, define
\begin{eqnarray}
\ff_\alpha u(\xi)&=&c_\alpha e^{-i\pi|\xi|^2\cot\alpha}\int_{\R}
u(t)e^{-i\pi |t|^2\cot\alpha}e^{-2i\pi t \xi/\sin\alpha}\mbox{d}t\notag\\
&=&c_\alpha e^{-i\pi|\xi|^2\cot\alpha}\ff[u(t)e^{-i\pi |t|^2\cot\alpha}](\xi/\sin\alpha).
\label{eq:deffrfo}
\end{eqnarray}
The last expression shows that $\ff_\alpha$ extends to a bounded operator on $L^2(\R)$.
Further $\ff_{\pi/2}=\ff$ the usual Fourier transform and we define $\ff_{2k\pi}=I$
the identity operator and $\ff_{(2k+1)\pi}f(x)=f(-x)$ then $\ff_{\alpha+\beta}=\ff_\alpha\ff_\beta$.
Finally, $\ff_\alpha$ is a unitary operator with $\ff_\alpha^*=\ff_{-\alpha}$.
One of the authors showed that $(\ff_\alpha)_{\alpha\in [0,2\pi)}$ does phase retrieval \cite[Proposition 4.2]{Ja2}. \\
An equivalent formulation is
that if for every time $t\geq 0$, the free Schr\"odinger evolution of $\ffi,\psi$
have same modulus $|e^{it\Delta}\ffi|=|e^{it\Delta}\psi|$ then $\ffi,\psi$ are equivalent up to a constant phase factor (this was conjectured in \cite{Vo}).

Proceeding from Conjecture \ref{conj:3unitaries} one may replace the constraint on the operators to be unitary by assuming them to be self-adjoint, and ask
\begin{quotation}
 \emph{Is there a triple of self-adjoint operators $(T_1,T_2,T_3)$ on $L^2(\mathbb{R})$ which does phase retrieval?}
\end{quotation}
Again, the point is that we want a triple of self-adjoint operators. For instance, Vogt \cite{Vo} stated (without proof) that the set of all rank one orthogonal projections does phase retrieval. An even smaller
set of rank one projections is sufficient. One may for instance take an orthonormal basis $(e_i)_{i\in \N}$
of $L^2(\mathbb{R})$ and then consider the rank one projections on the spaces $\mathrm{Span}(e_k)$, $k\in\N$,
$\mathrm{Span}(e_k+e_\ell)$, and $\mathrm{Span}(e_k+i e_\ell)$ $k\not=\ell\in\N$.
It is then easy to show that the family of associated orthogonal projections does phase retrieval.

On the other hand, shifting the focus towards the minimality of the employed operator family, without requiring
self-adjointness of the operators, the follwing was shown by one of the authors. Take $\gamma=e^{-\pi \cdot ^2}$ the Gaussian and $\alpha \in \mathbb{R} \setminus \pi\mathbb{Q}$, then
the pair $ T_1 \psi= \gamma \ast \psi$ and $T_2 \psi = \gamma \ast \ff_\alpha \psi$ 
does phase retrieval \cite[Proposition 4.1]{Ja2}.
Note that while $T_1$ is self-adjoint, $T_2$ is not. 

\begin{remark}
 Wright's conjecture has also attracted considerable amount of interest in the finite dimensional setting.
 Translating Conjecture \ref{conj:3unitaries} to $\mathbb{C}^d$ amounts to asking whether there exist three orthonormal bases 
 $(e_k^{(1)})_{k=1}^d, (e_k^{(2)})_{k=1}^d$ and $(e_k^{(3)})_{k=1}^d$, such that each and every vector $\psi\in\mathbb{C}^d$ is uniquely determined (up to multiplication by a unimodular constant) by the measurements
 \[
  |\langle e_k^{(j)},\psi\rangle|,\quad k=1,\ldots,d,~j=1,2,3.
 \]
This finite dimensional version has been disproved by  Morov and Perelomov \cite{MP} in the early 90s.
Further, one may relax the constraints and ask for a set of vectors $(e_k)_{k=1}^N$ such that $|\scal{e_k,\psi}|$,
$k=1,\ldots,N$ uniquely determines $\psi$ up to a unimodular constant.
Heinosaari, Mazzarella and Wolf \cite{HMW} proved that the minimal number of vectors is $\geq 3d+\alpha_d$
with $\alpha_d\to+\infty$ when $d\to+\infty$. \\
On the other hand, Mondragon and Voroninski \cite{MV}\footnote{This paper has not appeared yet.
However, a construction somewhat similar to our argument for rank-one projections gives an explicit family of
$5$ unitaries that lead to uniqueness up to a constant phase factor \cite{Go}.}
proved that for four ``generic'' orthonormal bases are enough to
 determine all $\psi$ up to a 
constant phase factor.
\end{remark}

\subsection{Contribution of this paper}
The purpose of this paper is to show that there are three simple and explicit masks such that the resulting coded diffraction patterns uniquely determine all univariate signals.
More precisely we will show the following.

\begin{theorem}\label{th:intro}
Let $\gamma_1=\gamma$ be the standard Gaussian, $\gamma(t) = e^{-\pi t^2}$,
and let $\gamma_2,\gamma_3$ be defined by
 \[
\gamma_2(t):=2\pi t\gamma_1(t), \qquad \gamma_3(t):= (1-2\pi t) \gamma_1(t).
 \]
\begin{enumerate}
\renewcommand{\theenumi}{\roman{enumi}}
\item \label{w1} Let $\ffi,\psi\in L^2(\R)$ be such that $|\ff[\gamma_1\ffi]|=|\ff[\gamma_1\psi]|$
and $|\ff[\gamma_2\ffi]|=|\ff[\gamma_2\psi]|$ then $\ffi$ and $\psi$ are equivalent up
to a constant phase factor and conjugation-reflection: $\psi=c\ffi$ or $\psi=c\ffi^*$
with $|c|=1$;

\item \label{w2} if we further assume that $|\ff[\gamma_3\ffi]|=|\ff[\gamma_3\psi]|$
then $\ffi$ and $\psi$ are equivalent up
to a constant phase factor; in other words, $(\ff\circ m_{\gamma_k})_{k=1}^3$ does phase retrieval.
\end{enumerate} 
\end{theorem}

\begin{remark}
The actual result can be extended in multiple ways. For instance, the function $\gamma_1$ can be replaced by 
$e^{-a|x|}$. We will also provide a second set of 3 operators that does phase retrieval. Finally, we will also show that the result can be extended to $L^2(\R^d)$ where we need $2d+1$ operators.
\end{remark}

\begin{remark}
 Note that the three masks $\gamma_1,\gamma_2,\gamma_3$ are real-valued, and consequently that $A_k=\ff \circ m_{\gamma_k} \circ \ff^\ast$, $k=1,2,3$ define self-adjoint operators on $L^2(\mathbb{R})$.
 It follows directly from Theorem \ref{th:intro} that $(A_1,A_2,A_3)$ does phase retrieval, hence we simultaneously solve the question posed earlier in Section \ref{subsec:motivation}.
\end{remark}

This result could be deduced from a result by McDonald \cite{Mc}: the main result of that paper can be summarized
as the identity and the derivation operator do phase retrieval (up to reflections)
when restricted to band-limited or even to narrow-banded functions. We will however give a more direct proof
and deduce our result from a bit more general facts. 
There are two possible strategies of proof. We could first establish \eqref{w1} and then deduce \eqref{w2}
from it. It turns out that this can be done in a more direct way using a simple lemma about analytic functions
(Lemma \ref{lem:uniqueness}). The proof of \eqref{w1} is a bit more evolved and uses a lemma from the second 
author. Deducing \eqref{w2} from it follows essentially the same lines as the ones used to directly establishing \eqref{w2}.

In a second part of this paper, we will move to the discrete setting. The operators we consider
have natural discrete analogues. More precisely, we will identify $\psi\in\C^d$ with an analytic
trigonometric polynomial $P_\psi$. The measurements we consider are then samples of $|P_\psi|$
and of $|P_\psi^\prime|$ the modulus of the derivative of $P_\psi$.
We will show that this requires $4d-2$ samples to lead to uniqueness (up to a constant phase factor)
and provide an example of non uniqueness with less samples. This is of course coherent
with the fact mentionned above that $3d$ phaseless measurements are not sufficient. However, it allows to
show the role of the sampling rate and explains why $3d$ phaseless measurements may not have been the right analogue
of Wright's conjecture in $\C^d$.

\medskip

The remainder of this paper is organized as follows: the next section is devoted to the continuous setting, followed by a section devoted to the discrete case. We conclude with an appendix to clarify the role of unitaries in
Wright's conjecture, mainly aimed to mathematicians without background on quantum mechanics.

\section{Continuous Level}

\subsection{Three Measurements}
We begin with an auxiliary result
which provides us with a uniqueness statement.
\begin{lemma}\label{lem:uniqueness}
Let $I\subseteq \mathbb{R}$ be an open interval and let $\mathcal{A}(I)$ denote the space of complex-valued analytic functions on $I$.
Then $F\in \mathcal{A}(I)$ is uniquely determined (up to multiplication by a unimodular constant) by $|F|^2$ and $F'\bar{F}$.
\end{lemma}

\begin{proof}
Suppose that $F,G\in \mathcal{A}(I)$ are such that $|F|^2=|G|^2$ and $F'\bar{F}=G'\bar{G}$.
We may assume w.l.o.g. that $|F|^2$ does not vanish identically. Therefore there exists a nonempty interval $I'\subseteq I$ such that $|F|^2=|G|^2$ has no zeros in $I'$.
Moreover, according to the assumption we have that 
\[
 (\log (G/F))' = (\log G)' - (\log F)' = \frac{G'\overline{G}}{G\overline{G}} - \frac{F'\overline{F}}{F\overline{F}} =0,
\]
i.e. $\log(G/F)$ is constant on $I'$. This implies that $G=\lambda F$ on $I'$ for some $\lambda\in\mathbb{C}$. Since $|G|^2=|F|^2$ we get that $\lambda$ must be unimodular. 
Finally, by analyticity the identity $G=\lambda F$ extends to all of $I$.
\end{proof}

We are now in position to prove the second part of the theorem:

\begin{proposition}\label{prop:P1}
Let $\gamma_1=\gamma$ be the standard Gaussian and let $\gamma_2,\gamma_3$ be defined by
 \[
\gamma_2(t):=2\pi t\gamma(t), \qquad \gamma_3(t):= (1-2\pi t) \gamma(t).
 \]
Let $\ffi,\psi\in L^2(\R)$ be two wave functions such that $|\ff[\gamma_j\ffi]|=|\ff[\gamma_j\psi]|$
for $j=1,2$ and $3$. Then $\ffi$ and $\psi$ are equivalent up
to a unimodular constant only.
\end{proposition}

\begin{proof}
First we define a pair of analytic function on the real line by 
$F:=\ff[\gamma_1\ffi]$ and $G:=\ff[\gamma_1\psi]$.

It is enough to show that $|\ff[\gamma_j\ffi]|=|\ff[\gamma_j\psi]|$ for $j=1,2,3$ implies that
$F=\lambda G$ with $|\lambda|=1$ since then $\gamma_1\ffi=\lambda\gamma_1\psi$. Then, as 
$\gamma_1$ does not vanish, we get that $\ffi$ and $\psi$ are equivalent up to a constant phase factor.

\smallskip

Now note that $F'=i\ff [\gamma_2 \ffi]$, and therefore that we have the identities 
\begin{equation*}
  |F|= |\ff[\gamma_1 \ffi]|, \qquad |F'|= |\ff[\gamma_2 \ffi]|, \qquad |F+iF'|=|\ff[\gamma_3 \ffi]|.
\end{equation*}
Thus, it remains to show that a function $F$ analytic on the real line, is uniquely determined given $|F|, |F'|$ and $|F+iF'|$. 
To see this, first compute 
$$
|F+i F'|^2 = |F|^2 + |F'|^2 +2\Re \left\{iF' \bar{F} \right\} = |F|^2 + |F'|^2 -2\Im \left\{F' \bar{F} \right\}.
$$
Together with 
$$
\Re\left\{F'\bar{F}\right\}= \frac12 \left( F'\bar{F}+\overline{F'}F\right) = \frac12 \left(|F|^2\right)'
$$
we get that 
\begin{equation}
\label{eq:pfw21}
 F'\bar{F} = \frac{1}{2} \left(|F|^2 \right)' + \frac{i}2 \left(|F|^2+|F'|^2- |F+iF'|^2 \right).
\end{equation}
It follows that $|F|=|G|$ and that $F'\overline{F}=G'\overline{G}$. Applying Lemma 
\ref{lem:uniqueness} yields the desired statement. 
\end{proof}

\begin{remark}
In the next section, we are going to prove that $|\ff[\gamma_j \ffi]|=|\ff[\gamma_j\psi]|$ for $j=1,2$
implies that $\ffi=c\psi$ or $\ffi=c\psi^*$. In this last case $G=c\overline{F}$.
But then $|\gamma_3 \ffi|=|\gamma_3\psi|$ reads 
$|F+iF'|^2=|\bar F+i\bar F'|^2$ which implies that $\Im \left\{F' \bar{F} \right\}=0$.
But then \eqref{eq:pfw21} simplifies to $F'\bar{F} = \frac{1}{2} \left(|F|^2 \right)'=G'\bar{G}$.
Again lemma \ref{lem:uniqueness} yields the desired statement. 

The direct proof given here is substentially simpler.
\end{remark}

\begin{remark}\label{rk:other}
The Gaussian $\gamma_1$ only plays a mild role here: 

-- it implies that $F=\ff[f\gamma_1]$ is holomorphic in a neighborhood of the real line
so that we may replace $\gamma_1$ by any function that is $O(e^{-a|x|})$ for some $a>0$,

-- $\gamma_1$ does not vanish on a set of positive measure so that $f$ is uniquely determined by $f\gamma_1$.

This shows that we could replace $\gamma_1$ by {\it e.g} $e^{-a|x|^\alpha}$, $a>0$, $\alpha\geq 1$.
\end{remark}

\subsection{Two Measurements}

We are now going to prove Theorem \ref{th:intro}\,\eqref{w1}.
The origin of our choice for the three operators comes from the work of Mc\,Donald \cite{Mc}
who characterized entire functions of finite order $F,G$ such that $|F|=|G|$
on the real line and $|F'|=|G'|$. Once one notices that $F=\ff[\gamma_1\psi]$
is entire of order 2, Mc\,Donald's result applies directly. We here propose another strategy of proof
that does not use the growth properties of $F$.
In fact, our arguments do not even require that the functions under consideration are entire
only that they are holomorphic in a neighborhood of the real line.

\begin{lemma}\label{lem:const}
Let $D\subseteq \mathbb{C}$ be a nonempty, open disk centered on the real line. 
Let $u,v$ be two smooth real valued functions
 such that $h(z):=u(x,y)+iv(x,y)$ is holomorphic in $z=x+iy\in D$. Assume that $u$ satisfies 
$$
u(x,0)=0  \quad \text{and} \quad
\partial_y u(x,0)=0 \quad \text{for\ all}\ x\in D \cap \mathbb{R}.
$$
Then $h=ia$ for some $a\in \mathbb{R}$.
\end{lemma}

\begin{proof}
 Without loss of generality we assume that $D$ is centered at the origin.
 We expand $h(z)=\sum_{k\in\mathbb{N}} a_k z^k$ as a power series. 
 The identity 
 \begin{equation*}
  0 = u(x,0) = \Re \left\{ \sum_{k\ge 0} a_k x^k \right\} = \sum_{k\ge 0} \Re \{a_k\} x^k
 \end{equation*}
 implies that each of the coefficients $(a_k)_{k\ge 0}$ is purely imaginary. Using that $\frac{\partial}{\partial y} h(x+iy) = \sum_{k\ge 1} a_k i k (x+iy)^{k-1}$ yields together with the second assumption that 
 \begin{equation*}
  0 = u_y(x,0) = \Re \left\{h_y(x,0) \right\}= \Re \left\{i \sum_{k\ge 1} a_k k x^{k-1}\right\} = -  \sum_{k\ge 1} \Im\{a_k\} k x^{k-1},
 \end{equation*}
 which implies that $\Im\{a_k\}= 0$ for $k\ge 1$.\\
 Therefore we have indeed that $a_k=0$ for $k\ge 1$ and $\Re\{a_0\}=0$, which yields the desired statement.
\end{proof}
Moreover, we require the following connection between
the complex derivative of an analytic function and the gradient of its modulus.
\begin{lemma}\label{lem:gradientmod}
 Let $D\subseteq \mathbb{C}$ be a domain in the complex plane and $h\in\mathcal{O}(D)$.
 Then it holds for all $z\in D$ with $h(z)\neq 0$ that 
 $$
 |\nabla |h|(z)| = |h'(z)|.
 $$
\end{lemma}

\begin{proof}
 This can be shown rather elementary using Cauchy-Riemann equations.
 See \cite[Lemma 3.4]{grohsstable19} for a proof.
\end{proof}

\begin{lemma}
\label{lem:der}
Let $F,G$ be two analytic functions and assume that for every $x\in\R$,
$|F(x)|=|G(x)|$ and $|F'(x)|=|G'(x)|$. Then there exists $c\in\C$ with $|c|=1$
such that either $G=cF$ or $G=c\bar{F}$.
\end{lemma}

This result can be found in \cite{Mc} for entire functions of finite order and in \cite{KJP}
for so-called wide-banded functions.

\begin{proof}
We resort to a nonempty open disk $D$ centered on the real line such that neither $F$ nor $G$ has any zeros in $D$. Note that such a disk always exists unless $F$ (or $G$) vanishes identically, in which case the statement is trivial.

By Lemma \ref{lem:gradientmod} we have for all $x\in D\cap \mathbb{R}$ that  
\begin{align*}
(\partial_y |G|)^2(x+i0) &= |G'(x+i0)|^2- (\partial_x |G|)^2(x+i0) \\ 
&= 
|F'(x+i0)|^2- (\partial_x |F|)^2(x+i0) = (\partial_y |F|)^2(x+i0),
\end{align*}
which implies that either 
\begin{center}
a) $\partial_y |G|=\partial_y |F|$ on $D\cap\mathbb{R}$ \quad or \quad
b) $\partial_y |G|=-\partial_y |F|$ on $D\cap\mathbb{R}$.
\end{center}
In case a) we consider 
$h:=\log(G/F)\in \mathcal{O}(D)$ and observe that due to $|F(x)|=|G(x)|$ for real $x$,
$$
\Re\{h\}(x+i0) = \log |G/F| (x+i0)  = 0, \quad x\in D\cap \mathbb{R}.
$$
Moreover, we have that 
\begin{align*}
 \partial_y \Re \{h\} (x+i0) &= \partial_y (\log |G| - \log |F|)(x+i0)\\
 &=  \left(\frac{\partial_y |G|}{|G|} - \frac{\partial_y|F|}{|F|}\right)(x+i0) = 0,\quad x\in D\cap \mathbb{R}.
\end{align*}
Applying Lemma \ref{lem:const} yields that $h=ic$ with $c\in\mathbb{R}$, which implies that 
$$
G/F = \exp h = e^{ic},
$$
and therefore $G=e^{ic}F$ as desired (by analyticity the identity holds on the full plane).\\
In case b) one considers $h:= \log G/\bar{F}$ and proceed similarly as in case a).
\end{proof}

We can now show \eqref{w1} of Theorem \ref{th:intro}:

\begin{proposition}\label{thm:uniqueness2msmts}
 Let $\gamma$ denote the standard Gaussian and let
 \begin{equation}
  \gamma_1(t):= \gamma(t) \quad\text{and}\quad  \gamma_2(t):= 2\pi t\gamma(t).
 \end{equation}
Assume that $\ffi,\psi\in L^2(\R)$ are two wave functions
such that $|\ff[\gamma_k\ffi]|=|\ff[\gamma_k\psi]|$ for $k=1,2$. Then
$\ffi$ and $\psi$ are equivalent up to a constant phase factor and conjugation.
\end{proposition}

\begin{proof}
Once more, we set $F:=\ff[\gamma_1\ffi]$ and $G:=\ff[\gamma_1 \psi]$
and note that these functions are analytic, even more so they extend to entire functions on the plane.
According to the assumption we have that 
\begin{equation}\label{id:modrealline}
|F(x+i0)|=|G(x+i0)| \quad\text{for all}~ x\in\mathbb{R},
\end{equation}
as well as
\begin{equation*}
 |F'(x+i0)| = |\ff[\gamma_2 f](x)| = |\ff[\gamma_2 g](x)| = |G'(x+i0)|, \quad x\in\mathbb{R}.
\end{equation*}
Lemma \ref{lem:der} then shows that $G=cF$ or $G=c\bar F$ which is equivalent to
$\ff^*\psi=c\ff^*\ffi$ or $\ff^*\psi(\xi)=c\overline{\ff^*[\ffi](-\xi)}$ since $\gamma_1=\gamma_1^*$ does not vanish. In turn, this is then equivalent to $\psi=c\ffi$ or $\psi=c\overline{\ffi}$.
\end{proof}

\begin{remark}
We have only used that $F$ is holomorphic in a neighborhood of the real line.
As for Remark \ref{rk:other}, the same proof thus applies if $\gamma$ is replaced by $e^{-a |x|^\alpha}$, $a>0$,
$\alpha\geq 1$.
Note the for $\alpha=1$, $F$ is only holomorphic in a strip.
\end{remark}

\subsection{A second family of three operators}

\begin{proposition}\label{thm:uniqueness2}
 Let $\gamma$ denote the standard Gaussian and let $a,b>0$ be such that $\dfrac{a}{b}\notin\Q$, and let
 \begin{equation} 
  \gamma_1(t):= \gamma(t),\quad  \gamma_2(t):= \sin(a\pi t)\gamma(t)
\quad\text{and}\quad \gamma_3(t):= \sin(b\pi t)\gamma(t)
 \end{equation}
Assume that $\ffi,\psi\in L^2(\R)$ are two wave functions
such that $|\ff[\gamma_k\ffi]|=|\ff[\gamma_k\psi]|$ for $k=1,2$ and $3$. Then
$\ffi$ and $\psi$ are equivalent up to a constant phase factor.
\end{proposition}

\begin{proof}
We again introduce $F=\ff[\gamma_1\ffi]$, $G=\ff[\gamma_1\psi]$
and notice that $F,G$ are entire functions of order 2
and that $|F(x)|=|G(x)|$.

Further,
using the standard fact that the Fourier transform of a modulation is the translation
of the Fourier transform and that $\sin\alpha=\dfrac{e^{i\alpha}-e^{-i\alpha}}{2i}$,
it is straightforward to see that

-- $|\ff[\gamma_2\ffi]|=|\ff[\gamma_2\psi]|$ if and only if $|F(x)-F(x-a)|=|G(x)-G(x-a)|$,

-- $|\ff[\gamma_3\ffi]|=|\ff[\gamma_3\psi]|$ if and only if $|F(x)-F(x-b)|=|G(x)-G(x-b)|$.

Applying twice the main result of \cite{Mc} we get that there exist two periodic
functions $W_a,W_b$ with repective period $a$ and $b$ and such that
both are meromorphic and continuous over $\R$ with $|W_a(x)|=|W_b(x)|=1$ for real $x$
and satisfy $G=W_aF=W_bF$. In particular, $W_a=W_b$ on $\R$ so that 
$W_a$ is both $a$ and $b$-periodic. But then for every $k,\ell\in\Z$ we have that 
$W_a(ak+b\ell)=W_a(0)$. As $a/b\notin\Q$, $\{ak+b\ell,k,\ell\in\Z\}$ is dense in $\R$
and by continuity of $W_a$ we get that $W_a$ is a constant of modulus one.
Finally, as $G=W_aF$ we get $\psi=W_a\ffi$ as claimed.
\end{proof}

\begin{remark}
The same proof applies if $\gamma$ is replaced by $e^{-\alpha |x|^p}$, $p\geq 1$.
Note the for $p=1$, $F$ would only be holomorphic in a strip and McDonald's result no longer applies.
In this case, one needs the extension of McDonald's result in \cite{KJP}.
\end{remark}

\begin{remark}
The condition $a/b\notin\Q$ is essential. Indeed, let $a\not=0$
and $b=\dfrac{p}{q}a$ with $p,q\in\Q$, $q\not=0$ and $\beta=\dfrac{q}{a}$.
We have chosen $\beta$ so that $e^{2i\pi\beta a}=e^{2i\pi\beta b}=1$.

Let $\ffi\not=0$ be smooth and compactly supported. Define $F=\ff[\gamma\ffi]$ and
define $\psi$ by
$$
\psi(t)=\frac{\gamma(t+\beta)}{\gamma(t)}\ffi(t+\beta)
=e^{-2\pi\beta t-\pi\beta^2}\ffi(t+\beta).
$$
A direct computation shows that
$$
G(x):=\ff[\gamma \psi](x)=\ff[\gamma(\cdot+\beta)\ffi(\cdot+\beta)](x)=e^{2i\pi\beta x}\ff[\gamma\ffi](x)
=e^{2i\pi\beta x}F(x).
$$
But then $|G(x)|=|F(x)|$,
$$
|G(x)-G(x+a)|=|e^{2i\pi\beta x}F(x)-e^{2i\pi\beta x}e^{2i\pi\beta a}F(x+a)|=|F(x)-F(x+a)|
$$
since $e^{2i\pi\beta a}=1$ and, replacing $a$ by $b$ in this computation, $|G(x)-G(x+b)|=|F(x)-F(x+b)|$.
The proof of Proposition \ref{thm:uniqueness2} then shows that
$|\ff[\gamma_k\ffi]|=|\ff[\gamma_k\psi]|$ for $k=1,2$ and $3$. Of course, $\psi$ is not a constant multiple of 
$\ffi$.
\end{remark}

\subsection{An extension to higher dimensions}

We will now give an extension to several variables. Let us start with a simple lemma
about several variable holomorphic functions.
We will make use of the following notation: for $j\in\{1,\ldots,d\}$ and $x=(x_1,\ldots,x_d)\in\R^d$,
write $x^{(j)}=(x_1,\ldots,x_{j-1},x_{j+1},\ldots,x_d)\in\R^{d-1}$.

\begin{lemma}\label{lem:multcomp}
Let $F,G$ be two non-zero holomorphic functions on $\C^d$ and assume that there
are functions $\ffi_1,\ldots,\ffi_d\,:\R^{d-1}\to\T$ such that, for every $j\in\{1,\ldots,d\}$
and every $x\in\R^d$, $F(x)=\ffi_j(x^{(j)})G(x)$.
Then there is a $c\in\T$ such that $F=cG$.
\end{lemma}

\begin{proof}
First, as $F$ is continuous and non-zero, there exists a ball $B(x_0,r)$ of $\R^d$
such that $F$ does not vanish on $B(x_0,r)$. Without loss of generality, we may assume that
$x_0=0$. Then as $|F(x)|=|\ffi_j(x^{(j)})||G(x)|=|G(x)|$, $G$ does also not vanish on $B(0,r)$
and therefore $\ffi_j(x^{(j)})=\dfrac{F(x)}{G(x)}$ for all $j$ does not depend on $j$.
But this implies that $\dfrac{F(x)}{G(x)}$ does not depend on any of the variables $x_1,\ldots,x_j$
on $B(0,r)$ and is thus a constant $c$ {\it i.e.} $F=c G$ on $B(0,r)$.
From the holomorphy of $F$ and $G$ we conclude that $F=cG$ on $\C^d$. 
\end{proof}

\begin{corollary}
Let $\gamma$ be the Gaussian on $\R^d$, $\gamma(t)=e^{-\pi|t|^2}$.
Let $f,g\in L^2(\R^d)$ be non-zero and such that $|\ff[\gamma f]|=|\ff[\gamma g]|$.
Assume further that one of the two following conditions are satisfied:

-- for all $j\in\{1,\ldots,d\}$, $|\ff[2\pi t_j\gamma f]|=|\ff[2\pi t_j\gamma g]|$
and $|\ff[(1-2\pi t_j)\gamma f]|=|\ff[(1-2\pi t_j)\gamma g]|$ on $\R^d$;

or

-- for all $j\in\{1,\ldots,d\}$, $|\ff[\sin \pi a_jt_j\gamma f]|=|\ff[\sin \pi a_j\gamma g]|$
and $|\ff[\sin \pi b_jt_j\gamma f]|=|\ff[\sin \pi b_j\gamma g]|$ on $\R^d$
wth $a_j,b_j>0$, $\dfrac{a_j}{b_j}\notin\Q$;

then there is a $c\in\T$ with $g=cf$.
\end{corollary}

\begin{proof} In both cases, consider $F=\ff[\gamma f]$ and $G=\ff[\gamma g]$
so that $F,G$ extend to holomorphic functions over $\C^d$. 

Let us consider the first set of hypothesis. Fix $\underline{\xi}=(\xi_2,\ldots,\xi_d)\in\R^{d-1}$ and denote by 
$$
f_{\underline{\xi}}(x)=\int_{\R^{d-1}}e^{-\pi|\underline{x}|^2}f(x,\underline{x})
e^{-2i\pi\scal{\underline{x},\underline{\xi}}}\,\mbox{d}\underline{x}
$$
and use a similar notation for $g$.

Let $\gamma_1$ be the Gaussian on $\R$ and $\ff_1$ be the $1$-variable Fourier transform, then
$$
|\ff_1[\gamma_1f_{\underline{\xi}}]|=|\ff[\gamma f](\xi_1,\underline{\xi})|=
|\ff[\gamma g](\xi_1,\underline{\xi})|
=|\ff_1[\gamma_1g_{\underline{\xi}}]|
$$
and similarily $|\ff_1[2\pi t\gamma_1f_{\underline{\xi}}]|=|\ff_1[2\pi t\gamma_1g_{\underline{\xi}}]|$
and $|\ff_1[(1-2\pi t)\gamma_1f_{\underline{\xi}}]|=|\ff_1[(1-2\pi t)\gamma_1g_{\underline{\xi}}]|$.
Proposition \ref{prop:P1} then implies that there exists $c(\underline{\xi})\in\T$
such that $f_{\underline{\xi}}(x)=c(\underline{\xi})g_{\underline{\xi}}(x)$.
Multiplying by $\gamma_1$ and taking Fourier transform, we then
get $F(\xi_1,\underline{\xi})=c(\underline{\xi})G(\xi_1,\underline{\xi})$
for every $\xi_1\in\R$ and every $\underline{\xi}\in\R^{d-1}$.

Doing the same for each variable, we see that the conditions of Lemma \ref{lem:multcomp}
are fullfilled. There is then $c\in\T$ such that $F=cG$ which implies that $f=cg$.

Replacing Proposition \ref{prop:P1} by \ref{thm:uniqueness2}, we get that the same is valid for the second set of conditions.
\end{proof}

Note that one can obtain the same result by imposing the first set of condition
for some coordinates and the second set for the others.

On the other hand, taking functions of the tensor form 
$$
f(x_1,\ldots,x_d)=f_1(x_1)\cdots f_d(x_d)
$$
it is easy to see that the full set of conditions is needed.

\section{Discretizations}

\subsection{Continuous derivative}

We now turn to a discrete setting. We consider $\psi=(\psi_0,\ldots,\psi_{N-1})\in\C^N$ which
can be identified with a the polynomial
$$
P_\psi(x)=\sum_{j=0}^{N-1}\psi_j e^{2i\pi jx}.
$$
It is crucial to notice that $P_\psi$ is a so-called analytic trigonometric polynomial
{\it i.e.} it has no negative frequencies. In particular $\bar P_\psi$ is not an analytic
trigonometric polynomial and can therefore not be of the form $P_\ffi$.
We will use this fact below.

\begin{remark}
Note that that if $M\geq N$,
$$
P_\ffi\left(\frac{k}{M}\right)=\sum_{j=0}^{N-1}\psi_j e^{2i\pi jk/M}
$$
is the $M$-dimensional discrete Fourier transform $\ff_M[\psi^{(M)}]$
where $\psi^{(M)}$ is the $0$-padded sequence $\psi^{(M)}=(\psi_0,\ldots,\psi_{N-1},0,\ldots,0)\in\C^M$.

Cand\'es {\it et al} proved that $\{|\ff_N[\psi](k)|,k=0,\ldots N-1\}$ together with the two difference sequences
$\{|\ff_N[\psi](k)-\ff_N[\psi](k-1)|,k=0,\ldots N-1\}$ and
$\{|\ff_N[\psi](k)-i\ff_N[\psi](k-1)|,k=0,\ldots N-1\}$ determine almost every $\psi\in\C^N$.

One can see $\ff_N[\psi](k)-\ff_N[\psi](k-1)$ as the discrete derivative of the sequence $\ff_N[\psi](k)$
and this result can thus be seen as a discrete analogue of Theorem \ref{th:intro}\,\eqref{w2}.
\end{remark}

Instead of a discrete derivative, let us first inverstigate what is happening if we consider the continuous
derivative, that is $P_\psi^\prime(x)=\dst2i\pi\sum_{j=0}^{N-1} j\psi_j e^{2i\pi jx}=P_{j\psi}$
with $j\psi=(0,\psi_1,\ldots,(N-1)\psi_{N-1})$. 
We are here asking whether for some $M\geq N$
\begin{equation}
 \label{eq:poltrig1}
\left\{\begin{matrix}
\abs{P_\ffi\left(\frac{k}{M}\right)}&=&\abs{P_\psi\left(\frac{k}{M}\right)}\\[9pt]
\abs{P_\ffi^\prime\left(\frac{k}{M}\right)}&=&\abs{P_\psi^\prime\left(\frac{k}{M}\right)}
  \end{matrix}\right.
   \quad \text{for }k=0,\ldots,M-1
 \end{equation}
implies $P_\ffi=\lambda P_\psi$ where $|\lambda|=1$ so that $\ffi=\lambda\psi$.
In other words, we are asking
whether 
$$
\{|\ff_M[\psi^{(M)}](k)|,|\ff_M[j\psi^{(M)}](k)|,k=0\ldots,M-1\}
$$
determines $\psi$ up to a constant phase factor.

\medskip

Now notice that $|P_\psi(x)|^2=P_\psi(x)\overline{P_\psi(x)}=\sum_{j,k=0}^{N-1}\psi_j\overline{\psi_k}e^{2i\pi (j-k)x}$ is a trigonometric polynomial of degree $N$.
We may write it in the form 
$$
|P_\psi(x)|^2=e^{-2i\pi (N-1)x}\sum_{\ell=0}^{2N-2}c_\ell e^{2i\pi \ell x}
$$
which shows that, up to the factor $e^{-2i\pi (N-1)x}$, $|P_\psi(x)|^2$ is a polynomial of degree $2N-2$
evaluated on the unit circle. Therefore it is determined by $2N-1$ distinct values. 
The same applies to $|P'|$.
For instance
$$
\abs{P_\psi\left(\frac{k}{2N-1}\right)},\abs{P_\psi^\prime\left(\frac{k}{2N-1}\right)},\quad k=0,\ldots,2N-2
$$
uniquely determine $|P_\psi|,|P_\psi^\prime|$.
We can then apply Lemma \ref{lem:der} wich then shows that for $M=2N-1$, \eqref{eq:poltrig1}
implies that there is a unimodular complex number $\lambda$
such that $P_\psi(x)=\lambda P_\ffi(x)$ or $P_\psi(x)=\lambda\overline{P_\ffi}(x)$.
As said above, $P_\ffi$ and $P_\psi$ are analytic trigonometric polynomials so that the
later case can not occur.
In conclusion

\begin{proposition}
Let $\psi,\ffi\in\C^N$ and assume that the corresponding trigonometric polynomials
satisfy 
 \begin{equation*}
\left\{\begin{matrix}
\abs{P_\ffi\left(\frac{k}{2N-1}\right)}&=&\abs{P_\psi\left(\frac{k}{2N-1}\right)}\\[9pt]
\abs{P_\ffi^\prime\left(\frac{k}{2N-1}\right)}&=&\abs{P_\psi^\prime\left(\frac{k}{2N-1}\right)}
  \end{matrix}\right.
   \quad \text{for }k=0,\ldots,2N-2
 \end{equation*}
then there exists $\lambda\in\C$ with $|\lambda|=1$ such that $\psi=\lambda\ffi$.
\end{proposition}

We are going to prove that this result is sharp in the sense that
\eqref{eq:poltrig1} for $M=2(N-1)$ is not sufficient for 
$\ffi,\psi$ to be identical up to a constant phase factor.
We start with $N=3$.

\begin{lemma}\label{lem:pols}
 Let $p(z)=z$ and $q(z)=\frac{z^2}2 + \frac{\sqrt{3}i}2$.
 Moreover, let $\Lambda=\{1,i,-1,-i\}$. \\
 Then it holds that $|p(\lambda)|=|q(\lambda)|$ and $|p'(\lambda)|=|q'(\lambda)|=1$ for all $\lambda\in \Lambda$.
\end{lemma}

\begin{proof}
 Obviously
  for $|z|=1$ it holds that $|p(z)|=|p'(z)|=|q'(z)|=1$. 
 Thus, it remains to check that $q(\lambda)$ is of unit modulus for $\lambda\in\Lambda$.
 Indeed we have that 
 \begin{align*}
  q(\pm 1) &= \frac{1+\sqrt{3}i}{2}=e^{\pi i/3},\\
  q(\pm i) &= \frac{-1+\sqrt{3}i}2 = e^{2\pi i/3}.
 \end{align*}
\end{proof}

\begin{proposition}\label{thm:counterexpls1}
 Let $N=2m+1$ be an odd integer $\ge 3$. 
 There exist $\ffi,\psi\in \mathbb{C}^N$ which are not equivalent up to a
 constant phase factor
 while
 \begin{equation*}
\left\{\begin{matrix}
\abs{P_\ffi\left(\frac{k}{2N-2}\right)}&=&\abs{P_\psi\left(\frac{k}{2N-2}\right)}\\[9pt]
\abs{P_\ffi^\prime\left(\frac{k}{2N-2}\right)}&=&\abs{P_\psi^\prime\left(\frac{k}{2N-2}\right)}
  \end{matrix}\right.
   \quad \text{for }k=0,\ldots,2N-3.
 \end{equation*}
\end{proposition}

\begin{proof}
 We use the polynomials from Lemma \ref{lem:pols} and define $\ffi$ to be the sequence of the coefficients of the polynomial $\tilde{p}(z):=p(z^m)$ and analogously, $\psi$ to consist of the coefficents of $\tilde{q}(z):= q(z^m)$. In other words $P_\ffi(x)=p(e^{2im\pi x})$ and $P_\psi(x)=q(e^{2im\pi x})$.
 
But then $P_\ffi\left(\frac{k}{2N-2}\right)=P_\ffi\left(\frac{k}{4m}\right)=
p(e^{ik\pi/2})$ and analogously, $P_\psi\left(\frac{k}{2N-2}\right)=q(e^{ik\pi/2})$. In particular
$\dst\abs{P_\ffi\left(\frac{k}{2N-2}\right)}=\abs{P_\psi\left(\frac{k}{2N-2}\right)}$ for
$k=0,\ldots,2N-2$.

On the other hand 
$$
P_\ffi^\prime(x)=2i\pi m e^{2im\pi x}p'(e^{2im\pi x})
\quad\text{and}\quad
P_\psi^\prime(x)=2i\pi m e^{2im\pi x}q'(e^{2im\pi x});
$$
thus, we get that 
$|P_\ffi^\prime(x)|=|P_\psi^\prime(x)|=2\pi m$ so that $\ffi,\psi$
satisfy the condition of the theorem.
 
Finally, since the number of non-zero coefficients $\ffi$ and $\psi$ are different, it is obvious that $\ffi$ and $\psi$ are not equivalent.
\end{proof}

In view of the result by Cand\'es {\it al} \cite{CSV} it seems natural to ask

\begin{question}
For which $M$ is it true that for almost every $\ffi\in\C^N$, every $\psi\in\C^N$ such that
\begin{equation}
 \label{eq:poltrig2}
\left\{\begin{matrix}
\abs{P_\ffi\left(\frac{k}{M}\right)}&=&\abs{P_\psi\left(\frac{k}{M}\right)}\\[9pt]
\abs{P_\ffi^\prime\left(\frac{k}{M}\right)}&=&\abs{P_\psi^\prime\left(\frac{k}{M}\right)}
  \end{matrix}\right.
   \quad \text{for }k=0,\ldots,M-1
 \end{equation}
is equivalent to $\ffi$ up to a constant phase factor?
In other words, for which $M$ is the set of vectors which possess nontrivial ambiguous solutions a set of measure zero?
\end{question}

\subsection{Discrete derivative}

In this section we consider again samples of $P_\ffi$\,: $u_k=P_\ffi\left(\dfrac{k}{M}\right)$
(seen as an $M$-periodic sequence)
and we ask whether $|u_k|$ and its discrete derivative $|u_k-u_{k-1}|$ determine $\ffi$
up to a constant phase factor.

This will follow from the following proposition:

\begin{proposition}\label{prop:poldfiff}
Let $P,Q$ be two polynomials of degree $\leq N$ and $0<b<2\pi/N$. Assume that 
for every $x\in\R$, 
$$
\left\{\begin{matrix}
|P(e^{ix})|=|Q(e^{ix})|\\
|P(e^{i(x+b)})-P(e^{ix})|=|Q(e^{i(x+b)})-Q(e^{i x})|
\end{matrix}\right.
$$
then there is a unimodular constant such that $P=cQ$.
\end{proposition}

\begin{proof}
The proof is divided into two steps. The first one is folklore
and the second part is an elaboration on a result by McDonalds \cite{Mc}.

Write $P(x)=\dst \alpha x^k\prod_{j=1}^K(x-x_j)$ and $Q(x)=\dst \beta x^l\prod_{j=1}^L(x-y_k)$
with $x_k,y_k\not=0$ and, without loss of generality $K\geq L$. Then
\begin{multline*}
|P(e^{2i\pi x})|^2=|\alpha|^2P(e^{2i\pi x})\overline{P(e^{2i\pi x})}
=|\alpha|^2\prod_{j=1}^K(e^{2i\pi x}-x_j)(e^{-2i\pi x}-\overline{x_j})\\
=|\alpha|^2e^{-2i\pi Kx}\prod_{j=1}^K(e^{2i\pi x}-x_j)(1-\overline{x_j}e^{2i\pi x})
\end{multline*}
while
$$
|Q(e^{2i\pi x})|^2=|\beta|^2e^{-2i\pi Lx}\prod_{j=1}^L(e^{2i\pi x}-y_j)(1-\overline{y_j}e^{2i\pi x}).
$$
It follows that $|P|=|Q|$ on the unit circle implies that if $\zeta=e^{2i\pi x}$
$$
|\alpha|^2\prod_{j=1}^K(\zeta-x_j)(1-\overline{x_j}\zeta)=|\beta|^2\zeta^{K-L}
\prod_{j=1}^L(e^{2i\pi x}-\zeta)(1-\overline{y_j}\zeta).
$$
This is an identity between two polynomials. As it is valid on the unit circle, it is valid over $\C$.
As a consequence, as the left hand side does not vanish at zero, so does the right hand side and $K=L$.
Further, the two polynomials have same zeros. The zeros of the left hand side (counted with multiplicity)
are $\{x_j,1/\bar x_j,j=1,\ldots,K\}$ and those of the right hand side are $\{y_j,1/\bar y_j,j=1,\ldots,K\}$
thus for every $j$, $y_j=x_j$ or $y_j=1/\bar x_j$, the reflection of $x_j$ with respect to the unit circle.
In particular, note that if $|x_j|=1$
then it is a common zero of $P$ and $Q$ and $1/\bar x_j=x_j$.

It follows that, up to reordering the zeroes, we may first list the zeros that are not reflected and
then those that are reflected:
$$
Q(x)=\beta z^l\prod_{j=1}^J(x-x_j)\prod_{j=J+1}^K(x-1/\overline{x_j}).
$$
In order to remove some ambiguities, note that one may have a pair of zeros of the form $\{x,1/\bar x\}$,
{\it i.e.} there are $j,k$ such that $x_j=x$ and $x_k=1/\bar x$.
Up to reordering the zeroes, we may assume that those $j,k$'s are $\leq J$. 
We can thus write $P(z)=\alpha z^kP_1(z)P_2(z)$ and $Q(z)=\beta z^lP_1(z)P_2^*(z)$ with 
$P_2(z)=\prod_{j=J+1}^K(x-x_j)$ and 
$P_2^*(z)=\prod_{j=J+1}^K(x-1/\overline{x_j})$. Moreover, assume that if $j,k\geq J+1$
then $x_j\not=1/\overline{x_k}$ since the corresponding terms can be put into $P_1$.

Our aim is to show that this factor $P_2$ is not present here. From now one we argue towards a contradiction by
assuming that there is at least one reflected zero, so that $P_2$ has at least one zero $x_J$.
Further, up to re-ordering the zeroes, we may assume that $|x_j|$ is non-decreasing for $j\geq J$.

From McDonald \cite{Mc} we know that
$Q(e^{ix})=W(x)P(e^{ix})$ with $W$ meromorphic, periodic of period $b$ with $|W(x)|=1$ for $x$ real,
continuous on the real line. The previous argument shows that
$$
W(x)=\frac{\beta}{\alpha}e^{i(l-k)x}\frac{\prod_{j=J+1}^K(e^{ix}-1/\overline{x_j})}{\prod_{j=J+1}^K(e^{ix}-x_j)}
$$
so that
$$
W(x+b)=\frac{\beta}{\alpha}e^{i(l-k)(x+b)}\frac{\prod_{j=J+1}^K(e^{ix}-e^{-ib}/\overline{x_j})}{\prod_{j=J+1}^K(e^{ix}-x_je^{-ib})}.
$$
But then $W(x)=W(x+b)$ implies that
$$
\prod_{j=J+1}^K(e^{ix}-1/\overline{x_j})\prod_{j=J+1}^K(e^{ix}-x_je^{-ib})
=e^{i(l-k)b}\prod_{j=J+1}^K(e^{ix}-e^{-ib}/\overline{x_j})\prod_{j=J+1}^K(e^{ix}-x_j)
$$
for every $x\in \R$ so that we have the identity between polynomials
$$
\prod_{j=J+1}^K(X-1/\overline{x_j})\prod_{j=J+1}^K(X-x_je^{-ib})
=e^{i(l-k)b}\prod_{j=J+1}^K(X-e^{-ib}/\overline{x_j})\prod_{j=J+1}^K(X-x_j)
$$
Therefore the sets of zeros $\{x_je^{-ib},j=J+1,\ldots,K\}\cup \{1/\overline{x_j},j=J+1,\ldots,K\}$
and $\{x_j,j=J+1,\ldots,K\}\cup \{e^{-ib}/\overline{x_j},j=J+1,\ldots,K\}$ are
equal (counting multiplicity).

Let $L\leq K\leq N$ be such that $|x_{J+j}|=|x_{J+1}|$ for $j=1,\ldots, L$.
If we had $\{x_{j+J}e^{-ib},j=1,\ldots,L\}=\{x_{j+J},j=1,\ldots,L\}$ with multiplicity then this set
would be invariant under multiplication by $e^{-ib}$. In particular, it contains $\{x_Je^{-ikb},k\in\Z\}$
but we have chosen $b<2\pi/N\leq 2\pi/L$ so $P_2$ would have more than $L$ zeros, a contradiction. Thus there is a $j,k$ such that $x_je^{-ib}=e^{-ib}/\overline{x_k}$.
that is $x_j=1/\overline{x_k}$, again a contradiction.

We are then left with $W(x)=\dfrac{\beta}{\alpha}e^{i(k-l)x}$ which is $b$-periodic. As $b<2\pi/N$
it follows that $k=l$. Thus $W$ is a constant of modulus $1$ and $Q=WP$ as claimed.
\end{proof}

Note that the argument also works if $b\in\R\setminus\Q\pi$. 

\begin{corollary}
Let $\ffi,\psi\in\C^N$ and assume that for $k=0,\ldots,2N-2$,
$$
\left\{\begin{matrix}
\abs{P_\ffi\left(\dfrac{k}{2N-1}\right)}=\abs{P_\psi\left(\dfrac{k}{2N-1}\right)}\\[9pt]
\abs{P_\ffi\left(\dfrac{k+1}{2N-1}\right)-P_\ffi\left(\dfrac{k}{2N-1}\right)}
=\abs{P_\psi\left(\dfrac{k+1}{2N-1}\right)-P_\psi\left(\dfrac{k}{2N-1}\right)}
\end{matrix}\right.
$$
then $\ffi,\psi$ are equivalent up to a phase factor.
\end{corollary}

\begin{proof}
As in the previous section, 
$\dst\abs{P_\ffi\left(\dfrac{k}{2N-1}\right)}, k=0,\ldots,2N-2$ fully defines $|P_\ffi|$ while 
$$\dst\abs{P_\ffi\left(\dfrac{k+1}{2N-1}\right)-P_\ffi\left(\dfrac{k}{2N-1}\right)}, k=0,\ldots,2N-2$$
fully determines $|P_\ffi(x+(2N-1)^{-1})-P_\ffi(x)|$.
Applying Proposition \eqref{prop:poldfiff} implies that there is $\lambda\in\C$ with $|\lambda|=1$
such that $P_\psi=\lambda P_\ffi$ which gives the result.
\end{proof}

We will now show that the result is false if we sample at a rate $1/(2N-2)$ instead of $1/(2N-1)$.

\begin{lemma}\label{lem:pols2}
 Let $p(z)=z$ and $q(z)=\frac{z^2+i}{\sqrt{2}}$. Moreover, let $\Lambda = \{1,i,-1,-i\}$. Then it holds that $|p(\lambda)|=|q(\lambda)|$ and $|p(\lambda)-p(i\lambda)|= |q(\lambda)-q(i\lambda)|$ for all $\lambda\in \Lambda$.
\end{lemma}

\begin{proof}
 This is easily checked by direct computation.
\end{proof}

\begin{proposition}
Let $N=2m+1$ be an odd integer $\ge 3$.
There exist signals $\ffi,\psi\in \mathbb{C}^ N$ which are not equivalent up to a phase factor
such that 
 for $k=0,\ldots,2N-3$,
 \begin{equation}
 \label{eq:last}
 \left\{\begin{matrix} \abs{P_\ffi\left(\frac{k}{2N-2}\right)} =  
 \abs{P_\psi\left(\frac{k}{2N-2}\right)}\\[9pt]
 \abs{P_\ffi\left(\frac{k}{2N-2}\right)- P_\ffi\left(\frac{k-1}{2N-2}\right)} 
=\abs{P_\psi\left(\frac{k}{2N-2}\right)- P_\psi\left(\frac{k-1}{2N-2}\right)} 
  \end{matrix}\right..
 \end{equation}
\end{proposition}

\begin{proof}
 We use a similar construction as in the proof of Proposition \ref{thm:counterexpls1} and define $\ffi,\psi$ to be the sequence of coefficients of $\tilde{p}(z):=p(z^m)$ and 
 $\tilde{q}(z):=q(z^m)$, respecitvely, where $p,q$ are the polynomials from Lemma \ref{lem:pols2}.
 Again, since the number of non-zero coefficients does not agree we find that $\ffi$ and $\psi$ are not equivalent.

To see that they satisfy \eqref{eq:last} note that from Lemma \ref{lem:pols2} we deduce that
\begin{eqnarray*}
\abs{P_\ffi\left(\frac{k}{2N-2}\right)}&=&\abs{p(e^{2i\pi m\frac{k}{4m}})}
=\abs{p(e^{i k\pi/2})}\\
&=&\abs{q(e^{i k\pi/2})} =\abs{P_\psi\left(\frac{k}{2N-2}\right)}
\end{eqnarray*}
while
\begin{multline*}
\abs{P_\ffi\left(\frac{k+1}{2N-2}\right)-P_\ffi\left(\frac{k}{2N-2}\right)}
=\abs{p(ie^{i k\pi/2})-p(e^{ik\pi/2})}\\
=\abs{q(ie^{i k\pi/2})-q(e^{ik\pi/2})}
=\abs{P_\psi\left(\frac{k+1}{2N-2}\right)-P_\psi\left(\frac{k}{2N-2}\right)}
\end{multline*}
which implies the claim.
\end{proof}

\section*{Appendix}
The purpose of this section is to provide some background on 
Wright's conjecture, as formulated in Conjecture \ref{conj:wright}.\\
We begin with introducing the main objects and notions appearing in quantum mechanics that we need here.
The space of all possible states of a quantum mechanical system is represented by $L^2(\mathbb{R})$. A state $\psi\in L^2(\mathbb{R})$ is also called a \emph{wave function}.
Two wave functions $\psi$ and $\ffi$ are considered equivalent if they agree up to multiplication by a unimodular constant.\\
Quantities of a system that can be measured are called \emph{observables} and represented by densely-defined self-adjoint operators on $L^2(\mathbb{R})$.
The expected value of the state $\psi\in D(A)$ in the observable $A$ is
defined as
$$
E_\psi(A)=\scal{\psi,A\psi}.
$$

The two following examples are essential in this paper. Let $u,v\in L^\infty(\R)$
be real valued. To $u$ and $v$ associate the following two observables\footnote{Note that we normalized the Fourier
transform $\ff$ so that it is unitary. Its adjoint is thus the inverse Fourier transform $\ff^*\ffi(x)=\ff^{-1}\ffi(x)=\ff\ffi(-x)$.}
$$
M_u\ffi=u\ffi\qquad\mbox{and}\quad \mm_v \ffi =\ff^*\bigl[v\ff[\ffi]\bigr].
$$
Then
$$
E_\psi(M_u)=\int_\R u(x)|\psi(x)|^2\,\mbox{d}x
$$
and 
$$
E_\psi(\mm_v)=\int_\R v(\xi)|\widehat{\psi}(\xi)|^2\,\mbox{d}\xi.
$$ 
Here, we keep the convention of notation in mathematics where the position variable is denoted by $x$
and the momentum variable is denoted by $\xi$ instead of $p$.

Let $\bb$ be the set of Borel subsets of $\R$.
It is then obvious that $|\psi(x)|$ is uniquely determined by
$$
\mathcal{E}_Q=\big\{E_\psi(M_{\mathbf{1}_B})\big\}_{B\in\bb}
$$
which is called the distribution of the state $\psi$ with respect to
position since
$$
\mathcal{E}_Q=\big\{\|\mathbf{1}_B(Q)\psi\|\big\}_{B\in\bb}
$$
where $\mathbf{1}_B(Q)$ are the spectral projections associated to the position operator.

On the other hand
$|\widehat{\psi}(\xi)|$ is uniquely determined by
distribution of the state $\psi$ with respect to momentum:
$$
\mathcal{E}_P:=\{E_\psi(\mm_v)\,:\ v=\mathbf{1}_B,\ B\mbox{ a Borel set}\}
=\{\|\mathbf{1}_B(P)\psi \|,\ B\mbox{ a Borel set}\}:=\mathcal{S}_P
$$
where $\mathbf{1}_B(P)$ are the spectral projections associated to the momentum operator.

In a footnote to the {\em Handbuch der Physik} article on the general principle
of wave mechanics \cite{Pa}, W. Pauli asked whether a wave function $\psi$ is uniquely determined (up to a constant phase factor) by one of the equivalent quantities

\begin{itemize}
\item the Pauli data $(|\psi|,|\widehat{\psi}|)$;

\smallskip

\item $\big\{\|\mathbf{1}_B(Q)\psi\|\big\}_{B\in\bb}$, $\big\{\|\mathbf{1}_B(P)\psi\|\big\}_{B\in\bb}$;

\smallskip

\item $\big\{E_\psi(M_{\mathbf{1}_B})\big\}_{B\in\bb}$, $\big\{E_\psi(\mm_{\mathbf{1}_B})\big\}_{B\in\bb}$.
\end{itemize}
The question can also be found {\it e.g.} in the book by H. Reichenbach \cite{Re} and in 
Busch \& Lahti \cite{BL}.

\smallskip

As mentionned in the introduction, it is known that in general the Pauli data does not uniquely determine
the state $\psi$ (up to a constant phase factor).

It is then natural to ask whether there exists a set of observables $(A_j)_{j\in J}$
(preferably including position and momentum or at least having a physical meaning) 
such that the associated sets built from spectral projections
$$
\mathcal{E}_j:=\big\{\|\mathbf{1}_B(A_j)\psi\|\big\}_{B\in\bb}
=\big\{E_\psi\bigl(\mathbf{1}_B(A_j)\bigr)\big\}_{B\in\bb}
$$
uniquely determine every state $\psi$.

Using the spectral theorem, to a self-adjoint operator $A_j$ we can associate
a unitary operator $U_j$ and a multiplication operator $M_j$ on a space $L^2(\mu_j)$
such that $A_j=U_j^*M_jU_j$. Then the data $\mathcal{E}_j$, $j\in J$
uniquely determine $|U_j\psi|$, $j\in J$. This then directly leads
to Wright's Conjecture \ref{conj:wright} and to its relaxation
\ref{conj:3unitaries}: {\sl find a set of measures $\mu_j$ and {\em unitary}
operators $U_j\,:L^2(\R^d)\to L^2(\mu_j)$ such that $|U_j\psi|=|U_j\ffi|$, $j\in J$,
implies that $\psi$ and $\ffi$ are equivalent up to a constant phase factor.}

A relaxed version is to find a set $\{T_j\}_{j\in J}$ of bounded self-adjoint (or even only
bounded) operators on $L^2(\R)$ such that $|T_j\psi|=|T_j\ffi|$, $j\in J$,
implies that $\psi$ and $\ffi$ are equivalent up to a constant phase factor.
The data $|T_j\psi|$ can also be interpreted as an expectation of the state
$\psi$ with respect to a family of observables. To be more precise,
to a bounded operator $T$, we may associate the self-adjoint operator
$A_u=T^*M_uT$ whith $u\in L^\infty(\R)$ real valued. Then
$$
\scal{\psi,A_u\psi}=\int_\R u(x)|T\psi(x)|^2\,\mbox{d}x
$$
so that $|T\psi|$ is uniquely determined by
$$
\ee_T:=\big\{E_\psi(T^*M_{\mathbf{1}_B}T)\big\}_{B\in\bb}.
$$
However, it does not seem possible to reformulate this family of measurements in terms of
spectral projections associated to a single self-adjoint operator. 

\section{Data availability}
No data has been generated or analysed during this study.

\section{Funding and/or Conflicts of interests/Competing interests}

The second author was supported by an Erwin-Schrödinger Fellowship (J-4523) of the Austrian Science Fund FWF.

The authors have no relevant financial or non-financial interests to disclose.

\bibliographystyle{plain}

\end{document}